\documentclass[12pt]{amsart}

\usepackage[english]{babel}
\usepackage[utf8]{inputenc}
\usepackage{array}
\usepackage{tikz}
\usepackage{setspace}
\usepackage{graphicx}
\usepackage{subcaption}
\usepackage{subfig}
\usepackage{amssymb}
\usepackage{amsmath}
\usepackage{mathtools}
\usepackage{bbm}
\usepackage{cancel}
\usepackage{thmtools} 
\usepackage{thm-restate}
\usepackage{verbatim}
\usepackage{pgfplots}

\numberwithin{equation}{section}
\numberwithin{figure}{section}

\newtheorem{theorem}{Theorem}[section]
\newtheorem*{theorem*}{Theorem}
\newtheorem{lemma}[theorem]{Lemma}
\newtheorem{proposition}[theorem]{Proposition}
\newtheorem{definition}[theorem]{Definition}
\newtheorem{example}[theorem]{Example}
\newtheorem{rem}{Remark}[section]

\usepackage[letterpaper,
            bindingoffset=0in,
            left=1.2in,
            right=1.2in,
            top=1in,
            bottom=1in,
            footskip=.25in]{geometry}
\newcommand{\sign}{\text{sign}}
\DeclareMathOperator{\interior}{int}
\DeclareMathOperator{\diam}{diam}
\renewcommand{\d}[1]{\ensuremath{\operatorname{d}\!{#1}}}

\pgfplotsset{compat=1.18}

\begin{document}

\title{Adapted Measures for Markov Interval maps}
\author{\L{}ukasz Krzywo\'{n}}
\date{\today}
\thanks{This work is supported by the National Science Foundation under Award No.~NSF DMS-2154378.}

\begin{abstract}

Adapted invariant measures, such as the natural area measure (Liouville), have a central place in the development of ergodic theory for billiards.
These measures ensure local Pesin charts can be constructed almost everywhere even in the nonuniformly hyperbolic setting.
Recently, for Sinai billiards satisfying certain conditions, the unique measure of maximal entropy has been shown to be adapted. 
However, not all positive entropy measures are.
To investigate the connection between entropy and adaptedness, we examine Markov interval maps with exactly one singularity.
We prove that a condition relating the entropy of the map and the ``strength" of the singularity determines if the measure of maximal entropy is adapted with respect to this singularity.
We also show that under a H\"{o}lder condition, recurrence of the singularity is necessary to have nonadapted invariant measures.
\end{abstract}

\maketitle

\section{Introduction}\label{Intro}

In the setting of hyperbolic dynamics with discontinuities, a standard construction of the stable and unstable manifolds at a point requires a condition that has been termed ``adaptedness".
Roughly speaking, an invariant measure is adapted if it does not give too much weight to neighborhoods of the discontinuities.
As an example, for Sinai billiards, the discontinuities are also one sided singularities in the sense that derivative of the billiard map is unbounded near the singularity.\footnote{We will use ``singularity" only for when the derivative of a map is unbounded near a point, whereas discontinuity will have its common meaning. Definition \ref{singularity} defines which points are singularities and Definition \ref{adaptedmeasure} specifies which measures are adapted in the Markov interval map setting.}
For this dynamical system the natural invariant area measure (Liouville) is adapted \cite{KS}.
Baladi and Demers in \cite{BD} have shown, under a condition of sparse recurrence to singularities, that the measure of maximal entropy (MME) for a Sinai billiard map is unique and adapted.
Work in progress by Climenhaga and Day suggests that uniqueness extends to all Sinai billiards, but without sparse recurrence, it is unknown whether the MME is adapted or not.
Thus, it is natural to ask under what conditions the MME is adapted.
To this end we examine Markov interval maps with one singularity and show that adaptedness with respect to this singularity is related to the topological entropy and a H\"{o}lder exponent bounding the ``strength" of the singularity.
See Remark \ref{rem2} for a connection to the sparse recurrence condition in \cite{BD}.

\begin{restatable*}{theorem}{one}
\label{thm:one}
Let $I = [0,1]$ and let $f\colon I \to I$ be a piecewise $C^1$ uniformly expanding transitive Markov map.\footnote{These conditions imply that there exists a unique measure of maximal entropy for $(I,f)$.} Suppose there exists $\delta >0$ and $\alpha >1$ such that the interval map is defined by $f(x)=x^{1/\alpha}$ on $[0,\delta]$ and has no other singularities. 
    Then the MME for $(I,f)$ is adapted with respect to $0$ if and only if $h_\mathrm{top}(f) > \log(\alpha)$.
\end{restatable*}

Climenhaga, Demers, Lima, and Zhang also analyze adapted and nonadapted measures for billiard maps in \cite{CDLZ}.
They construct a nonadapted measure with positive entropy for billiard maps with a periodic orbit that has a single grazing point.
This periodicity (and in particular recurrence) is essential for their argument.
To explore what happens without recurrence, observe that if an interval map has a nonrecurrent singularity, then the closer a point is to the singularity, the longer its orbit stays away from a neighborhood of the singularity.
This seems to indicate that any invariant measure is adapted, which, if the singularity is not too ``strong", is true.

\begin{restatable*}{theorem}{three}
\label{thm:three}
 Let  $f\colon I \to I$ be a piecewise $C^1$ uniformly expanding Markov map such that 
 \begin{enumerate}
     \item $f$ has a singularity, $p^+ \in B'$, see \eqref{endpoints}, and no other singularities,
     \item $p^+$ is not periodic with respect to $(\tilde{I}, \tilde{f})$, see \eqref{tildemap} and Definition \ref{persing}, 
     \item $f$ is H\"{o}lder continuous at $p^+$, see Definition \ref{holder}.
 \end{enumerate}
 Then, every $f$-invariant measure on $I$ is adapted with respect to $p$.
\end{restatable*}
In Example \ref{nonpolynonadapt} we construct an interval map in this setting that satisfies condition $(1)$ and $(2)$, but not $(3)$, such that the MME is nonadapted.

In the mid 1980s, Katok and Strelcyn \cite{KS} modified Pesin theory to the case of uniform hyperbolicity with singularities.
Lima and Sarig \cite{LS} applied this work to Poincar\'{e} sections for 3-dimensional flows that are ``adapted" to a given invariant probability measure.
In our setting, we are using ``adaptedness" of invariant measures for discrete time systems as introduced by Lima and Matheus in \cite{LM} which we define below (Definition \ref{adaptedmeasure}).
However, as explained below in Remark \ref{p-adaptedness}, we are only treating adaptedness with respect to singularities.

\subsection{Acknowledgments}

I would like to thank my advisor Vaughn Climenhaga for helping me with uncountably many drafts, Fan Yang and Stefano Luzzatto for helpful comments on the geometric Lorenz models, and Dmitry Dolgopyat for interesting questions and comments regarding dimensions of measures.

\subsection{Outline}
In Section \ref{PreDef}, we will define our terms and how we handle some of the technical pieces of Markov interval maps.

In Section \ref{results}, we will state the main results and make some remarks on how they relate to billiards.

In Section \ref{coding}, we will construct a coding for Markov interval maps and identify one of the main tools, Gibbs bounds on the MME.

In Section \ref{proofs}, we will prove our main results.

 Section \ref{examples} is split into three parts.
 First, we will construct some example interval maps that highlight the limitations of the main results and need for certain conditions.
Second, we will show how the results relate to a dimension of ergodic invariant measures.
Finally, we will apply our main results to interval maps induced by geometric Lorenz models.

\section{Preliminary Definitions and Examples}\label{PreDef}

The following discussion draws heavily from \cite[Chapter 4.3]{PY}.
Let $f\colon I \to I$ be a piecewise $C^1$ uniformly expanding map on a closed bounded interval $I \coloneq [a,b] \subset \mathbb{R}$. 
That is, there exists $\lambda > 1$ and a set
\begin{equation}\label{subint}
  B = \{a=x_0<x_1< ...< x_m=b\} \subset I
\end{equation}
 defining subintervals, $I_i \coloneq [x_i, x_{i+1}]$, such that 
\begin{enumerate}
    \item[(A)] on each $I_i$, there exists a continuous monotonic $\tilde{f}_i \colon I_i \to I$, satisfying
    \begin{equation}\label{extension}
        f|_{\interior(I_i)} = \tilde{f}_i|_{\interior(I_i)},
    \end{equation}
    \item[(B)] for each $x_i$, $f(x_i) = \lim_{x \to x_i^+}f(x)$ or $f(x_i) = \lim_{x \to x_i^-}f(x) ,$
    \item[(C)] (Expanding) $f|_{\interior(I_i)}$ is $C^1$ with $|(f|_{\interior(I_i)})'| \geq \lambda$ for each $ i \in \{0, ..., m-1\}$,
    \item[(D)] (Markov) $\tilde{f}_i(I_i) = \bigcup_{j \in V(i)} I_j$ for some $V(i) \subset \{0, ..., m-1\}$.
\end{enumerate}

Because $f$ is not a continuous map, we do not have a ``topological" entropy as it is usually defined.
However, we do have a whole space, $\mathcal{M}_f$, of $f$-invariant Borel probability measures.
For $\mu\in \mathcal{M}_f$, let $h_{\mu}$ denote the measure theoretic entropy of $(I,f,\mu)$.
By recalling the variational principle for a continuous map, we can define the topological entropy to be the supremum of the measure theoretic entropies, let   $h_\mathrm{top}(f) \coloneq \sup_{\mu \in \mathcal{M}_f}h_{\mu}.$

\begin{definition}
   If there exists an $f$-invariant Borel probability measure $\mu$ on $I$ such that $h_{\mu}=h_\mathrm{top}(f)$ then $\mu$ is called a measure of maximal entropy (MME).
\end{definition}

 \begin{figure}
    \centering
   \begin{tikzpicture}

\begin{axis}[axis lines = left, axis equal image, width = 7cm, xlabel = \(x\), xtick={0,.5,1},
    ytick={0,.5,1}, clip = false,]

    \addplot [domain=.5:.6, samples=80, color=blue,]{-1.58*sqrt(x-.5)+.5};
    \addplot [color=blue, mark = none, smooth,] coordinates{(.6, 0) (1,1)};
    \addplot[color=blue, mark=none,dashed,] coordinates{(0,0) (1,1)};
    \addplot[color=blue, mark=none,smooth,] coordinates{(0,1) (.25,0)};
    \addplot[color=blue, mark=none,smooth,] coordinates{(.25,0) (.5,.5)};
    \node[circle,fill,inner sep=1.5pt] at (axis cs:.5,.5) {};
    \addplot[color=blue, mark=none,dashed,] coordinates{(.5,.5) (.5,0)};
\end{axis}
\end{tikzpicture}
    \caption{Periodic and singular but not a periodic singularity}%
    \label{fig:counter}%
\end{figure}
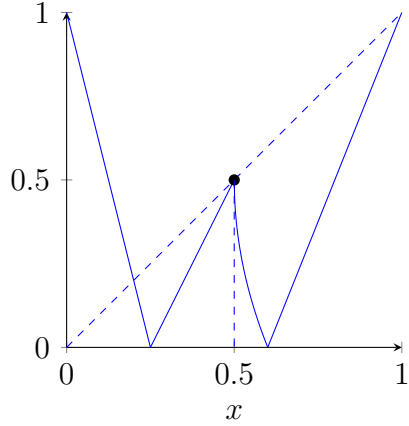
The map $f$ is piecewise monotonic but it is not necessarily orientation preserving.
In order to define periodic or non-periodic singularities, it will matter if $f$ changes the orientation of intervals. 
For example, the graph in Figure \ref{fig:counter} shows a continuous map with a fixed point that has a singularity.
However, this fixed point is actually a nonrecurrent singularity.
This is because the map does not send any right neighborhood of the singularity into another right neighborhood of the fixed point.
In fact, any invariant measure for this map is adapted with respect to the singularity by Theorem \ref{thm:three}.
In order to formulate our results precisely, we will keep track of the orientation of iterates of one-sided neighborhoods of the singularity.

Thus, we will define a related dynamical system $(\tilde{I}, \tilde{f})$.
Recall our set of endpoints, $B$ (\ref{subint}), and let $\tilde{B} = \bigcup_{n \geq 0} f^{-n}(B)$.
Now consider the set $\tilde{B} \times \{-1,1\}$.
With the notation of one-sided limits in mind, we will write $x^+ \coloneq (x,1)$ and $x^- \coloneq (x,-1)$ for $x \in \tilde{B}$.
Define $\tilde{B}^- = \{ x^- : x \in \tilde{B} \setminus \{a,b\}\}$ and  $\tilde{B}^+ = \{ x^+ : x \in \tilde{B}\setminus \{a,b\}\}$.
Thus, we may combine the disjoint sets to define
\[ \tilde{I} = (I \setminus \tilde{B}) \cup\{a,b\} \cup \tilde{B}^- \cup \tilde{B}^+.\] 
Let $\tilde I$ be given the order topology for the lexicographic order.
Let $\tilde{f}\colon \tilde{I} \to \tilde{I}$ be defined such that $\tilde{f}$ is continuous, piecewise monotonic, and for all $x \in (I\setminus \tilde{B}) \cup \{a,b\},$ $\tilde f(x) = f(x).$
This uniquely determines the map $\tilde f$ to satisfy the following description.

\begin{equation}\label{tildemap}
\begin{cases} 
      \tilde{f}(x) = f(x) & x \in (I\setminus \tilde{B}) \cup \{a,b\} \\
      \tilde{f}(x_i^-) = \left(\lim_{x \to x_i^-}f(x), -1 \cdot\sign(\tilde{f}_i')\right) & 1 \leq i \leq m-1 \\
      \tilde{f}(x_i^+) = \left(\lim_{x \to x_i^+}f(x), \sign(\tilde{f}_i')\right) & 1 \leq i \leq m-1 \\
      \tilde{f}(c^-) = \left(f(c), -\sign(\tilde{f}_{j(c)}'(c))\right) & c^- \in   \tilde{B}^- \\
      \tilde{f}(c^+) = \left(f(c), \sign(\tilde{f}_{j(c)}'(c))\right) & c^+ \in   \tilde{B}^+.
   \end{cases}
\end{equation}
Let $\pi_1 \colon\tilde{I}\to I$ be defined by 
\begin{equation}\label{proj}
\begin{cases} 
      \pi_1(x) = x & x \in (I\setminus \tilde{B}) \cup \{a,b\} \\
      \pi_1(x^-) = x & x^- \in   \tilde{B}^- \\
      \pi_1(x^+) = x & x^+ \in   \tilde{B}^+.
   \end{cases}
\end{equation}
Also, if we remove the countable set $\tilde{B}$ from each, then $(\tilde{I}\setminus \tilde{B},\tilde{f})=(I\setminus\tilde{B},f)$.
Thus, there is a correspondence between positive entropy measures on the two systems.
Hence, $h_\mathrm{top}(f)=h_\mathrm{top}(\tilde{f})$.
Since our construction uniquely determined $(\tilde{I},\tilde{f})$ from $(I,f)$, we will freely pass between the two when convenient.
We now consider the set
\begin{equation}\label{endpoints}
    B' = \{a=x_0, x_1^-,x_1^+, ..., x_{m-1}^-,x_{m-1}^+,x_m=b\} \subset \tilde{I}.
\end{equation}
in order to define singularities.
\begin{definition}\label{singularity}
    We call a point $p^+ \in B'$ a  singularity of $f$ if $\limsup_{x \to p^+}|f'(x)| = \infty $ and a point $p^- \in B'$ a singularity of $f$ if $\limsup_{x \to p^-}|f'(x)| = \infty $.
\end{definition}
We will only consider maps with one singularity, $x_i^+$. Note that in this convention $x_0^+ = x_0 =a$.
The reason for only considering left endpoints is that $(I,f)$ is conjugate to $(-I, -f)$ by the homeomorphism $h(x) =-x$, which would change a right endpoint into a left endpoint.

\begin{definition}\label{persing}
    Let $p^+ \in B'$ be a singularity of $f$. We say $p^+$ is a periodic singularity if there exists an $n \in \mathbb{N}$ such that $\tilde{f}^n(p^+) = p^+$.
    The minimum such $n$ is the period of the singularity.
    Otherwise, $p^+$ is a non-periodic singularity.
\end{definition}
If $f$ is Markov, endpoints must go to endpoints. Thus, the orbit of a singularity must be eventually periodic.
We now define adaptedness with respect to the singularity, $p^+$.

\begin{definition}\label{adaptedmeasure}
    Suppose $f \colon I \to I$ has a singularity $p^+ \in B'$. Let $I_p^+ = \{x \in I : x>p\}$ and define $b \colon I \setminus\{p\} \to \mathbb{R}$ by $b(x) = \mathbbm{1}_{I_p^+}(x)|\log(x-p)|$. 
    An $f$-invariant Borel probability measure $\mu$ is called adapted with respect to $p$, or $p$-adapted, if $\mu(\{p\}) = 0$ and $\int_{I_p^+} b(x)\d \mu(x) < \infty$, and $p$-nonadapted otherwise. 
\end{definition}

\begin{rem}\label{p-adaptedness}
    Let us highlight that our definition of adaptedness is with respect to a single point, not the whole set of discontinuities.
    If $\mu$ is $p$-adapted, then by the Birkhoff Ergodic Theorem, \[\lim_{n \to \infty}\frac{b\circ f^n(x)}{n}  =0\] for $\mu$ almost every $x \in I$.
    That is, for $\mu$ almost every $x$ and for all $\epsilon >0$ there exists $c_x>0$ such that $f^n(x) \notin [p, p+c_xe^{-\epsilon n})$ for all $n \in \mathbb{N}$.
    It is also important to note that, as Example \ref{nonsing} shows, singularities are not necessary for a measure to be nonadapted with respect to a given discontinuity point.
    For the rest of this work, ``adapted" means ``adapted with respect to $p$" or ``$p$-adapted".
    We may at times write the latter for emphasis.
\end{rem}

\begin{figure}
  \centering
  \begin{subfigure}{.3\linewidth}
    \centering
    \begin{tikzpicture}
\begin{axis}[
    axis lines = left, 
    axis equal image, 
    width = 6cm,
    xlabel = \(x\),
    clip = false,
    xtick={0,.5,1},
    ytick={0,.5,1},
]
    \addplot [domain=0:.11, samples=100, color=blue, ]{sqrt(x)};
    \addplot [ color=blue, mark = none, smooth,] 
        coordinates{(.11, {sqrt(.11)}) (.5,1)};
    \node[circle,fill,inner sep=1.5pt] at (axis cs:0,0) {};
    \addplot[domain=.5:1, samples=100, color=blue,]{2*x-1};
\end{axis}
\end{tikzpicture}
    \caption{Fixed Singularity}
  \end{subfigure}%
  \hspace{1.7em}% 
  \begin{subfigure}{.3\linewidth}
    \centering
    \begin{tikzpicture}

\begin{axis}[axis lines = left, axis equal image, width = 6cm, xlabel = \(x\), xtick={0,.5,1}, ytick={0,.5,1}, clip = false,]

    \addplot [domain=.3:.46, samples=100, color=blue,]{sqrt(x-.3)};
    \addplot [color=blue, mark = none, smooth,] coordinates{(.46, {sqrt(.16)}) (1,1)};
    \addplot[color=blue, mark=none,dashed,] coordinates{(0,0) (1,1)};
    \addplot[color=blue, mark=none,smooth,] coordinates{(0,.3) (.3,1)};
    \node[circle,fill,inner sep=1.5pt] at (axis cs:.3,0) {};
    \node[circle,fill,inner sep=1.5pt] at (axis cs:0, .3) {};
    \addplot[color=blue, mark=none,dashed,] coordinates{(0,.3) (.3,.3)};
    \addplot[color=blue, mark=none,dashed,] coordinates{(.3,0) (.3,.3)};
\end{axis}
\end{tikzpicture}
    \caption{Periodic Singularity}
  \end{subfigure}%
  \hspace{1.7em}% 
  \begin{subfigure}{.3\linewidth}
    \centering
    \begin{tikzpicture}
\begin{axis}[axis lines = left, axis equal image, width = 6cm, xlabel = \(x\), xtick={0,.5,1},
    ytick={0,.5,1}, clip = false,]
    \addplot [domain=.5:.61, samples=100, color=blue,]{sqrt(x-.5)};
    \addplot [ color=blue, mark = none, smooth,] coordinates{(.61, {sqrt(.11)}) (1,1)};
    \node[circle,fill,inner sep=1.5pt] at (axis cs:.5,0) {};
    \addplot [domain=0:.5, samples=100, color=blue, ]{2*x};
\end{axis}
\end{tikzpicture}
    \caption{Nonperiodic Singularity}
  \end{subfigure}
  \caption{Interval maps with singularity marked with a dot.}
  \label{fig:examples}
\end{figure}
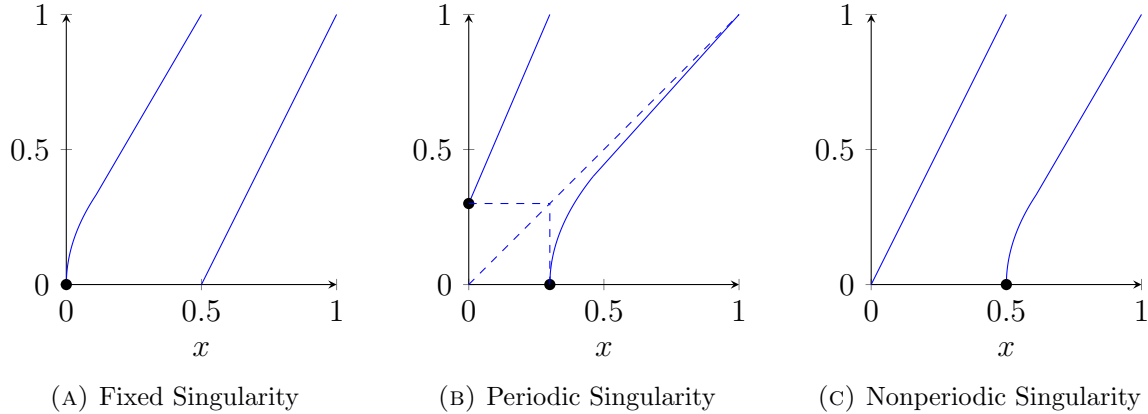

Figure \ref{fig:examples} shows examples of the interval maps we consider.
The first graph depicts a dynamical system that is conjugate to the doubling map but has a fixed singularity at $0$ coming from $f(x) = \sqrt{x}$ for small $x$.
The MME for this system is nonadapted by Theorem \ref{thm:one}. 
The second graph shows a related case where the singularity is periodic. 
The first example is a special case of Theorem \ref{thm:one}, and the second example is a special case of Theorem \ref{thrm2}.
Finally, the third graph shows a case when the singularity is not periodic. 
This case is conjugate to the first case, but due to the nonrecurrence of the singularity, we expect invariant measures to be adapted.
This is because the orbits of points sufficiently close to the singularity are near the orbit of the singularity which would limit the amount of time a nonwandering orbit stays near the singularity.
We will see in Theorem \ref{thm:three} that every invariant measure is adapted if $f$ is H\"{o}lder continuous near the singularity (Definition \ref{holder}).
Thus, for non-periodic singularities of the form $(x-p)^{1/\alpha}$, not only is the MME adapted, but so is every other invariant measure.

\begin{definition}
    \label{holder}
    We say $f$ is H\"{o}lder continuous near a singularity $p^+ \in B'$ if the map $\tilde{f_j}$ \eqref{extension} corresponding to $\tilde{I_j} = [p, x_{j+1}]$ is H\"{o}lder continuous.
\end{definition}

\section{Results}\label{results}
Given conditions on the strength of the singularity for an interval map and its entropy, we are able to determine whether or not the MME is adapted.

\one

When $h_\mathrm{top}(f) \neq \log(\alpha)$, Theorem \ref{thm:one} is a special case of Theorem \ref{thrm2}.
We prove the case $h_\mathrm{top}(f) = \log(\alpha)$ in Section \ref{Pthrm1}.
This value, $\alpha$, is a parameter controlling the ``strength" of the singularity or steepness of the map $f$ near the singularity.
This parameter can be recovered from the derivative of $f$ near the singularity in the following way.
For the class of functions defined in Theorem \ref{thm:one}, $\frac{\log(f'(x))}{\log(x)} \approx \frac{1}{\alpha}-1$ for $x \approx 0$.
In fact,
\[ \lim_{x \to 0^+} \frac{\log(f'(x))}{-\log(x)} = 1- \frac{1}{\alpha} \in (0,1).\]
Let us enlarge the class of maps we are considering to allow the singularity to not be exactly $x^{1/\alpha}$ and allow the singularity to be at any $p \in [0,1)$ and let us define
\begin{equation}\label{loglim}
       L(x) = \frac{\log((f^n)'(x))}{-\log(x-p)},\quad \quad   \overline{\beta} = \limsup_{x \to p^+}L(x),  \quad \quad \underline{\beta} = \liminf_{x \to p^+}L(x).
   \end{equation}
Under the stronger assumption of Theorem \ref{thm:one}, that $f(x) = x^{1/ \alpha}$ on $[0,\delta]$, we have $\overline \beta = \underline \beta = 1 - \alpha^{-1} \in (0,1).$
In general, $0 \leq \underline \beta \leq \min(1,\overline \beta)$.
Indeed, suppose by way of contradiction that $\underline\beta > 1$.
Hence, there exists a $\delta >0$ such that for every $x \in (0, \delta)$, we have $\frac{\log(f'(x))}{-\log(x)} > 1$, so $f'(x) > \frac{1}{x}$.
But then, 
\[f(\delta)-f(0) = \int_{(0,\delta)}f'(x)\d x = \infty,\] which contradicts the fact that $f$ is an interval map.
Thus, $\underline\beta \leq 1$.
The values $\underline \beta,\overline \beta$ determine an interval, describing the strength or steepness of the singularity.
We now formulate the main result of this work.

\begin{theorem}{\label{thrm2}}
   Let $I = [0,1]$ and let $f \colon I \to I$ be a piecewise $C^1$ uniformly expanding Markov map with a periodic singularity, $p^+ \in B'$, of period $n$.  
   Let $\underline\beta, \overline\beta$ be defined as in \eqref{loglim}.
   We also assume that $f$ has no other singularities.
   Then, for $I_J$, the transitive component of $f$, see Section \ref{coding}, containing the singularity and $h$ the topological entropy of $(I_J,f)$, the following hold.
    \begin{enumerate}
        \item If $\overline \beta <1$ and $h > -\frac{1}{n}\log(1-\overline{\beta})$, then the MME for $(I_J,f)$ is $p$-adapted.
       \item If $\underline \beta <1$ and $h < -\frac{1}{n}\log(1-\underline{\beta})$, then the MME for $(I_J,f)$ is $p$-nonadapted.
       \item When $\underline\beta = 1$ the MME for $(I_J,f)$ is $p$-nonadapted.
   \end{enumerate}
\end{theorem}
Theorem \ref{thrm2} is proved in Section \ref{Pthrm2}.
One interpretation shared by Theorems \ref{thm:one} and \ref{thrm2} is that if we have a periodic singularity, then the MME will be adapted if the dynamical system has enough entropy.
What determines how much entropy is ``enough" depends on the singularity and is captured by the quantities in \eqref{loglim}.
There is one difference, however. 
In Theorem \ref{thm:one} there is no indeterminacy, but in Theorem \ref{thrm2} there is.
If the topological entropy is in the middle interval, as follows
\begin{equation}\label{indeterm}
-\frac{1}{n}\log(1-\underline\beta) \leq h \leq -\frac{1}{n}\log(1-\overline\beta), 
\end{equation}
then Theorem \ref{thrm2} is indeterminate.
That is, additional information about $f$ would need to be known to determine whether or not the MME for $(I_J, f)$ is $p$-adapted.
Example \ref{eqadapt} constructs a map where $h=\frac{1}{n}\log(\alpha)$ and the MME is adapted, showing that indeterminacy is possible in the setting of Theorem \ref{thrm2}.

\begin{rem}\label{threshold}
    Since $f$ has no other singularities, the chain rule implies that we may replace $(f^n)'(x)$ in \eqref{loglim} with $f'(x)$ to achieve the same values $\underline\beta$ and $\overline\beta$.
    Also, if for some $\delta>0$ we have
    \[f(x)= f(p)+(x-p)^{1/\alpha}\text{ for }x \in (p,p+\delta),\]
    then $\underline\beta=\overline\beta = 1-\frac{1}{\alpha}$.
\end{rem}
\begin{rem}\label{rem2}
    If $\underline\beta=\overline\beta = 1- \frac{1}{\alpha}$, there is a similarity between this result and the sparse recurrence condition in \cite{BD}. Their condition is written as $h>s_0 \log(2)$, where $s_0$ bounds how often orbits can be nearly grazing.
    If we identify $s_0$ with $1/n$, since the singularity has period $n$, then statement $(1)$ in Theorem \ref{thrm2}, $h > \frac{1}{n}\log(\alpha)$, becomes $h>s_0 \log(\alpha)$.
    Choosing $\alpha=2$ corresponds to the setting of dispersing billiards (see \cite[Section 2.4]{BD}).
\end{rem}

Last we consider when the singularity is non-periodic.

\three

Theorem \ref{thm:three} is proved in Section \ref{Pthrm3}.
In Example \ref{nonpolynonadapt} we show that there exists an interval map that is not H\"{o}lder continuous near a non-periodic singularity and whose MME is nonadapted.
\begin{rem}
    Recalling the discussion in Remark \ref{rem2}, it is also of interest whether a relationship between $\alpha$ and $h$ can determine the adaptedness of an invariant measure and in particular the MME when the singularity is a non-periodic recurrent point. 
    In the case of Markov interval maps, there is no such possibility as every endpoint of a subinterval is either periodic to itself or preperiodic to an orbit that does not contain itself.
    In the case of non-Markov interval maps, it may be possible to describe some conditions on the rate of recurrence of the singularity that could give a result like Theorem \ref{thrm2} or Theorem \ref{thm:three}.  
\end{rem}

\section{Coding the Markov Partition}\label{coding}
In this section  we recall some standard known results\footnote{Most of these results can be found in \cite{PY}.} leading up to Lemma \ref{Gibbs}, which is a key ingredient of our proofs.
We first construct a coding for a Markov map $f\colon I \to I$ satisfying the conditions described at the beginning of Section $\ref{PreDef}$ with a periodic singularity, $p^+$.
Recall the notation, that $I$ is partitioned into subintervals, $I_i$, with disjoint interiors and for each $i \in S \coloneq \{0,...,m-1\}$ there is a collection of consecutive indices $V(i)$ such that $\tilde{f}_i(I_i) = \bigcup_{j \in V(i)}I_j$. 
To identify transitive components, we define a partial order, $\precsim$, on $S$ by 
\begin{equation}
i \precsim j \text{ if there exists an }n \in \mathbb{N}\text{ such that }f^{-n}(\interior(I_j)) \cap I_i \neq \emptyset.
\end{equation}
Given $i \in S$, let $J(i) = \{ j \in  S :  i \precsim j \text{ and } j \precsim i\}$.
This set could be empty, but taking $j^*$ such that $p = x_{j^*}$, the periodicity of $p^+$ guarantees that there exists an $n \in \mathbb{N}$ such that $f^{-n}(\interior(I_{j^*}))\cap I_{j^*} \neq \emptyset$. 
Let $J \coloneq J(j^*)$ and $I_{J} \coloneq \bigcup_{j \in J} I_j$.
Thus, $f|_{I_{J}}$ is transitive.
Recall $\tilde{B} \coloneq \bigcup_{n \geq 0}f^{-n}(B),$ and let
\begin{equation}\label{Iprime}
   I' \coloneq  I_J \setminus \tilde{B}.  
\end{equation}
Let the elements of $J$ be labeled $\{ j_1, j_2, ..., j_{|J|}\}$.
We will take the symbols in $J$ to be our alphabet and the sequence space will be a closed subset of  $J^{\mathbb{N}_0}$.
Let us put a metric on $J^{\mathbb{N}_0}$ by defining for $\omega, \nu \in J^{\mathbb{N}_0}$
\[d(\omega, \nu) = 2^{-\min(\{n : \omega_n \neq \nu_n\})}.\]
The topology on $J^{\mathbb{N}_0}$ induced by this metric has a basis of cylinders.
A cylinder is defined by $[w] = \{\nu : \nu_i = w_i \}$, where $w$ is a word of finite length formed by concatenating symbols from $J$.
With this topology and metric, $J^{\mathbb{N}_0}$ is a compact metric space.

Let us construct a function on the set of points whose orbits do not intersect the set of endpoints (\ref{Iprime}), $c \colon I' \to J^{\mathbb{N}_0}$, by the following procedure.
For $x \in I'$, $c(x) \in J^{\mathbb{N}_0}$ is the sequence satisfying $c(x)_n = j$ where $f^n(x) \in \interior(I_j) $ for every $n \in \mathbb{N}_0.$
By the definition of $I'$, this function is well defined.

 To describe $c(I')$ we construct a $|J| \times |J|$  $0$-$1$ matrix $A=(a_{ik})$ where, for $j_i,j_k \in J$, we put  $a_{ik} = 1$ if $f(\interior(I_{j_i}))\cap \interior(I_{j_k}) \neq \emptyset$, and $0$ otherwise.
 This adjacency matrix identifies which sequences in $J^{\mathbb{N}_0}$ will be admissible.
 That is, an element $\omega \in J^{\mathbb{N}_0} $ is admissible  if $a_{\omega_i\omega_{i+1}}=1$ for all $i \in \mathbb{N}_0$.
Denote by $\Sigma_A^+ \subset J^{\mathbb{N}_0}$ the sequences admitted by the adjacency matrix $A$. 
Since $f|_{I_J}$ is transitive, for each $i,j \in \{1, ..., |J|\}$, there exists an $n$ such that $(A^n)_{ij} \neq 0$.
A matrix with this property is called irreducible.
\begin{lemma}\label{lemma1}
    The map $c\colon I' \to c(I')$ is a homeomorphism.
\end{lemma}

\begin{proof}
By definition, $c$ is surjective.
Suppose $x, y \in I'$ and $c(x) = c(y) = \omega$. Then $|f^n(x) - f^n(y)| \leq \diam{I_{\omega_n}} \leq 1$ for all $n \in \mathbb{N}$. 
Since $f$ is uniformly expanding, this implies that there is a $\lambda > 1 $ such that for all $n \in \mathbb{N}$, $|x - y| \leq  \lambda^{-n}$. 
Thus, $x=y$, so $c$ is injective.

We next show that $c$ is continuous. 
Note that the sets $[w] \cap c(I')$, where $w$ is an admissible word, form a basis in the subspace topology.
By construction, $\bigcap_{i=0}^n f^{-i}(\interior(I_{w_i})) \neq \emptyset$. 
Take an open interval with endpoints in $I'$, $(a,b) \subset \interior(I_j)$ for some $j \in J$. 
Let $r \in \mathbb{N}_0$ be the least element such that $f^r(a)$ and $f^r(b)$ are in different intervals. Then, $c((a,b)\cap I') = [w]\cap c(I')$ for a cylinder of length $r$.
This association also shows that $c^{-1}$ is continuous.
\end{proof}

\begin{lemma}
The image of $I'$ under $c$ is dense.  That is, $\overline{c(I')} =  \Sigma_A^+$.
\end{lemma}

\begin{proof}
    By construction, $c(I') \subset \Sigma_A^+$, so $\overline{c(I')} \subset  \Sigma_A^+$.

    Let $\omega \in \Sigma_A^+ \setminus c(I')$. Consider the word $w^k = (w_0, ..., w_k)$. 
    By the definition of $A$, \[D_k=\bigcap_{i=0}^k f^{-i}(\interior(I_{w_i})) \neq \emptyset.\]
    Choose any $x \in I' \cap D_k$. Thus, for all $k \in \mathbb{N}$ we have $x_k \in I'$ and $c(x_k) \in [w^k]$. 
    Hence, $d(\omega,c(x_k)) \leq |J|^{-1-k}. $
    Therefore, $\lim_{k \to \infty}c(x_k) = \omega$, so $\Sigma_A^+ \subset \overline{c(I')}.$  
\end{proof}

Define $\sigma_A \colon \Sigma_A^+ \to \Sigma_A^+$ to be the left shift operation. That is, $\sigma_A(\omega)_i = \omega_{i+1}$.

We define a semiconjugacy $\pi$ from $(\Sigma_A^+, \sigma_A)$ to $(I_J, f)$ as follows.
For each $i \in J$ we require $\tilde{f}_i \circ \pi = \pi \circ \sigma_{A}|_{[i]}$, where $\sigma_{A}|_{[i]}$ means we restrict to sequences that start with $i$.
If $\omega \in c(I')$, then defining $\pi(\omega) = c^{-1}(\omega)$ is sufficient by Lemma \ref{lemma1}.
To extend our definition of $\pi$ to $\Sigma_A^+$, we define, for $\omega \in \Sigma_A^+$ and $n \in \mathbb{N}_0$, $f^n_{\omega} = \tilde{f}_{\omega_{n-1}} \circ \tilde{f}_{\omega_{n-2}} \circ ... \circ \tilde{f}_{\omega_0}$.
Now, define $\pi(\omega) = \bigcap_{n=0}^{\infty}(f^n_{\omega})^{-1}I_{\omega_n}.$
To show $\pi$ is well defined on $\Sigma_A^+$ we must show $(1)$ this intersection is nonempty and $(2)$ this intersection contains only one element.

$(1)$ Since $\tilde{f}_i$ is continuous, $\tilde{f}_i^{-1}I_j$ is closed for all $i,j \in J$.
Also, by the definition of $A$, $\bigcap_{n=0}^{k}(f^n_{\omega})^{-1}I_{\omega_n} \neq \emptyset$ for all $k \in \mathbb{N}_0$. 
Thus, by the finite intersection property, $\bigcap_{n=0}^{\infty}(f^n_{\omega})^{-1}I_{\omega_n} \neq \emptyset.$

$(2)$ Suppose $x \neq y$ and $x,y \in \bigcap_{n=0}^{\infty}(f^n_{\omega})^{-1}I_{\omega_n}$. Without loss of generality suppose $x<y$. Then, for all $n \in \mathbb{N}_0$, $f^n([x,y]) \in I_{\omega_n}$, which contradicts the condition that $f$ is uniformly expanding.

The map, $\pi$, is continuous on $c(I')$ by Lemma \ref{lemma1}. By the argument for $(2)$ we also have that $\pi$ is continuous on $\Sigma_A^+$ because if $\omega_k \to \omega$ in $\Sigma_A^+$, for any $\epsilon >0$, there exits $n \in \mathbb{N}_0$ such that $\diam(f_{\omega}^{-n}(I_i)) < \epsilon$ for each $i \in J$.
Thus, for $k > n$, $|\pi(\omega) - \pi(\omega_k)| < \epsilon$, so $\lim_{k \to \infty} \pi(\omega_k) = \pi(\omega)$.

Our dynamical system $(\Sigma_A^+, \sigma_A)$ is a subshift of finite type (SFT).
Thus, there is a unique MME given by the Parry measure (see \cite[Section 4.4.c]{KH}).\footnote{In this reference, they assume $A$ is primitive (irreducible and aperiodic) because they use a weaker form of the Perron--Frobenius Theorem, but the Parry measure is the same.}
The Parry measure is a Markov measure defined as follows.
By the Perron--Frobenius Theorem for irreducible non-negative matrices, of which $A$ is a member because $(I_J,f)$ is transitive, we have a Perron--Frobenius simple eigenvalue, $\lambda >0$, and corresponding left and right eigenvectors $u,v$ with positive entries normalized such that $\langle u,v\rangle = 1$.
Define the probability vector $p = (u_1v_1, ...,u_mv_m)$ and the stochastic matrix $P_{ij} = \frac{A_{ij}v_j}{\lambda v_i}$.
Then the Parry measure is the $(P,p)$-Markov measure, $\mu$, given by 
\[\mu([j_{i_1}...j_{i_n}]) = p_{i_1}P_{i_1i_2}...P_{i_{n-1}i_n} = u_{i_1}v_{i_1}\frac{v_{i_2}}{\lambda v_{i_1}}...\frac{v_{i_m}}{\lambda v_{i_{m-1}}} = u_{i_1}v_{i_m}\lambda^{-m+1}\]
for any cylinder $[w]$ defined by the $A$-admissible word $w = (j_{i_1}...j_{i_n})$.
It is shown separately in \cite{KH} that $\log(\lambda)$ is both the topological entropy of $(\Sigma^+_A,\sigma)$ and the measure theoretic entropy of $\mu$.
Thus, the MME, $\mu$, satisfies Gibbs bounds.
That is, there exist constants $c_1,c_2 >0$ such that for any $A$-admissible word, $w$, of length $n$ we have 
  \begin{equation}\label{GibbsBound}
   c_1 e^{-nh} \leq \mu([w]) \leq c_2e^{-nh}.
  \end{equation}

We can define a measure $\mu_f$ on $I_J$ by requiring $\mu_f(U) = \mu(\pi^{-1}(U))$ for any open set $U \subset I_J$.
Thus, for a $\mu_f$-measurable function $b$ on $I_J$,
\begin{equation}
    \int_{I_J} b(x)\d \mu_f(x) = \int_{\Sigma_A^+} b \circ \pi (\omega) \d \mu(\omega).
\end{equation}
This measure, $\mu_f$, is the MME of $(I_J, f)$.
The preceding description directly implies the following lemma.
\begin{lemma}\label{Gibbs}
   Let $A \subset \interior(I_j)$ for some $j \in J$ and $\{w^n\}$ be a sequence of words indexed by $n$ such that $w^n$ has length $n$. If there exist constants $L_n, R_n \in \mathbb{R_+}$ depending on $A$ and $\{w^n\}$ such that 
   \begin{enumerate}
       \item $\pi^{-1}(A) \subset \bigcup_{n=1}^{\infty}[w^n]$ 
       \item for all $x \in A$ such that $\pi^{-1}(x) \in [w^n]$, $L_n \leq b(x) \leq R_n$,
   \end{enumerate}
   then for the constant $c_2>0$ from \eqref{GibbsBound}
\[
 \int_A b(x)\d \mu_f(x) \leq \sum_{n=1}^{\infty}R_n c_2e^{-nh_f}. 
\]
If we also have that the interiors of the cylinders $[\omega^n]$ are disjoint then for the constant $c_1\geq 0$ from \eqref{GibbsBound}
\[
 \sum_{n=1}^{\infty}L_n c_1e^{-nh_f} \leq \int_A b(x)\d \mu_f(x). 
\]
\end{lemma}

\section{Proofs of Theorems}\label{proofs}

\subsection{Proof of Theorem \ref{thm:one}}\label{Pthrm1}

    As mentioned above, Theorem \ref{thm:one} is a special case of Theorem \ref{thrm2} if $h_\mathrm{top}(f) \neq \log(\alpha)$.
    Thus, we need only show the case $h_\mathrm{top}(f) = \log(\alpha)$.

\begin{proposition}\label{propeq}

    Let $I = [0,1]$ and let $f\colon I \to I$ be a piecewise $C^1$ uniformly expanding transitive Markov map. Suppose there exists $\delta >0$ and $\alpha >1$ such that $f(x)=x^{1/\alpha}$ on $[0,\delta]$ and $f$ has no other singularities. 
    If  $h_\mathrm{top}(f) = \log(\alpha)$, then the MME, $\mu$, for $(I,f)$ is nonadapted.
    
\end{proposition}
\begin{proof}
The map $f$ has a fixed singularity at $0$, is semiconjugate to an SFT on $m$ elements for some $m \in \mathbb{N}$, and the MME, $\mu$, is a Parry measure with entropy $h.$
Let the subintervals for $(I,f)$ be labeled by $\{0, ...,m-1\}$.
By transitivity, there exists a $j \in \{1,...,m-1\}$ such that the set $B_n \coloneq \pi([0^nj])\cap(0,\delta)$ is not empty for any $n \in \mathbb{N}$.
If $x \in B_n$, then $f^n(x) < \delta \leq f^{n+1}(x)$. 
Since for $x$ in this region, $f^n(x) = x^{\alpha^{-n}}$, we have
$x^{\alpha^{-n}}<\delta$.
Thus, $ \log(x)\alpha^{-n} < \log(\delta)$, so $b(x) = |\log(x)| > |\log(\delta)|\alpha^n.$
Therefore, since $h = \log(\alpha)$, by Lemma \ref{Gibbs}, we have 
\[\int b(x) \d \mu(x)  \geq  \sum_{n=1}^{\infty}\alpha^{-n-1}\alpha^n|\log(\delta)| = \infty .\]
Thus, the MME for $(I,f)$ is nonadapted.
\end{proof}

\subsection{Proof of Theorem \ref{thrm2}}\label{Pthrm2}
Let $I = [0,1]$ and let $f \colon I \to I$ be a piecewise $C^1$ uniformly expanding Markov map with a periodic singularity, $p^+$, of period $n$. 
Recall \eqref{loglim}, 
   \begin{equation*}
       L(x) = \frac{\log((f^n)'(x))}{-\log(x-p)},\quad \quad   \overline{\beta} = \limsup_{x \to p^+}L(x),  \quad \quad \underline{\beta} = \liminf_{x \to p^+}L(x).
   \end{equation*}
   We also assume that $f$ has no other singularities.
   Then, for $I_J$ the transitive component of $f$ containing the singularity and $h$ the topological entropy of $(I_J,f)$, we will show the following.
    \begin{enumerate}
        \item If $\overline \beta <1$ and $h > -\frac{1}{n}\log(1-\overline{\beta})$ the MME for $(I_J,f)$ is adapted.
       \item If $\underline \beta <1$ and $h < -\frac{1}{n}\log(1-\underline{\beta})$ the MME for $(I_J,f)$ is nonadapted.
       \item When $\underline\beta = 1$ the MME for $(I_J,f)$ is nonadapted.
   \end{enumerate}
To simplify the proof, let us define $g \colon [-p,1-p] \to [-p,1-p]$ by  $g(x) = f^n(x+p)-p $.
Let $B_{f^n}$ be the set of endpoints of subintervals for $f^n$ and $B_g$ be the set of endpoints of subintervals of $g$.
Let $J'$ be the transitive component of $(I,f^n)$ that contains the interval with $p$ as a left endpoint and $K$ be the transitive component of $(I-p, g)$ that includes the interval, $I_0$, with $0$ as a left endpoint.
We collect some facts:
\begin{enumerate}
    \item[(A)] $g(0) = 0$,
    \item[(B)] $g$ is uniformly expanding,
    \item[(C)] $g'(x) = (f^n(x+p))'$ for $x \in (0, \delta)$,
    \item[(D)] $ \underline{\beta} \in [0,1]$ and $\overline\beta \geq 0$,
    \item[(E)] $B_g = B_{f^n} - p$,
    \item[(F)] the coding for $(I_K, g)$ is isomorphic by relabeling to the coding for $(I_{J'}, f^n)$,
    \item[(G)] the entropy of $(I_K,g)$ is $nh$.
\end{enumerate}

\begin{lemma}\label{adapted}
    The MME for $(I_J,f)$ is adapted if and only if the MME for $(I_K, g)$ is adapted.
\end{lemma}

\begin{proof}
    There is an entropy-preserving correspondence between invariant Borel probability measures on $(I_J, f)$ and $(I_K, g)$ that give zero measure to $\bigcup_{i\geq 0} f^{-i}(B)$ and $\bigcup_{i\geq0} g^{-i}(B_g)$.
    By (F), it is sufficient to consider $(I_{J'}, f^n)$ and $(I_J,f)$.
    We will construct the correspondence by lifting to two-sided SFTs.
    First, we can code $(I_J,f)$ and $(I_J,f^n)$ with SFTs $(\Sigma_A, \sigma)$ and $(\Sigma_{A^n},\sigma)$.
    There exists a submatrix $A'$, of $A$,  such that $(\Sigma_{A'}, \sigma)$ is a coding for $(I_{J'},f^n).$
    There is a natural measure theoretic isomorphism between the $f$-invariant Borel probability measures on the one-sided and two-sided SFT's $(\Sigma_A, \sigma)$.
    Note that because $(I_J,f)$ is transitive, $(\Sigma_A,\sigma)$ is transitive, so $A$ is an irreducible matrix.

    We can use the cyclic structure of transitive SFTs \cite[Section 4.5]{LinMar} to decompose $\Sigma_A$ into disjoint sets $\Sigma_{A_i}$, $1 \leq i \leq n$, labeled such that $A'=A_1$.
    In fact, since $f^n(p)=p$, $\sigma$ cyclically permutes the sets $\Sigma_{A_i}$ and $(\Sigma_{A'},\sigma^n)$ is mixing.
    This gives a correspondence $\phi$ between $\sigma^n|_{\Sigma_{A'}}$-invariant measures $\nu$ and $\sigma|_{\Sigma_A}$-invariant measures $\mu$.
    For a word $w$, admitted by $A'$, $\phi^{-1}(\mu)([w]) = n\mu([w])$.
    For a word $w$, admitted by $A$, 
    \[\phi(\nu)([w]) = \frac{1}{n} \sum_{i=0}^{n-1} \nu\left(\sigma^i([w])\cap \Sigma_{A'}\right).\]
    Thus, given an $\sigma^n|_{\Sigma_{A'}}$-invariant measure $\nu$, the induced measure on $(I_J, f)$ will be adapted if and only if the measure induced by $\phi(\nu)$ on $(I_{J'},f^n)$ is adapted.
\end{proof}

Now we prove Theorem \ref{thrm2} for $(I_K, g)$.
By the construction in Section \ref{coding}, we have a coding for $(I_K, g)$ given by $(\Sigma_A, \sigma)$ and a semiconjugacy $\pi \colon \Sigma_A \to I_K$.
Let $A = (a_{ij})$.
Let the symbols be $\{0,...,n-1\}$ and $I_0$, as above, be the interval in $I_K$ coded by $0$ and with $0$ as a left endpoint.
Then, by (D), we have that for all $ \epsilon > 0 $ there exists a $\delta_1>0$ such that $\delta_1 \in I_0$ and if $x \in (0, \delta_1),$ 
   \[\underline{\beta}-\epsilon < \frac{\log(g'(x))}{-\log(x)} < \overline{\beta}+\epsilon.\]
   This implies
  \[-(\underline{\beta}-\epsilon)\log(x) < \log(g'(x)) < -(\overline{\beta} + \epsilon)\log(x),\]
so
\begin{equation}\label{bound}
   x^{-(\underline{\beta}-\epsilon)} < g'(x) < x^{-(\overline{\beta}+\epsilon)}. 
\end{equation}
Let
\begin{equation}\label{M}
    M \coloneq \min\{k \in \mathbb{N} : g^k(\delta_1) \notin I_0\}.
\end{equation}
Recall that $\underline\beta \leq 1$.

    \subsubsection{Proof of statement (1) of Theorem \ref{thrm2}}
 Suppose $\overline\beta < 1$ and $h > -\frac{1}{n}\log(1-\overline{\beta})$. Then, $1-\overline\beta > e^{-h_g}$. Let $\epsilon>0$ be chosen such that $1-(\overline\beta+\epsilon) > e^{-h_g}$.
   By the Mean Value Theorem and the bounding inequality (\ref{bound}), for every $x \in (0, \delta_1)$ there exits a $ c \in (\frac{x}{2},x)$ such that
   \[ \frac{g(x) -g(\frac{x}{2})}{\frac{x}{2}} = g'(c) < c^{-(\overline\beta+\epsilon)} < \left(\frac{x}{2}\right)^{-(\overline\beta+\epsilon)}.\]
   Let $s = 1-(\overline\beta+\epsilon)$, so $e^{-h_g} < s < 1,$ and thus
   \[g(x) < g\left(\frac{x}{2}\right) +\left(\frac{x}{2}\right)^s < g\left(\frac{x}{4}\right)+\left(\frac{x}{2}\right)^s +\left(\frac{x}{4}\right)^s < ... < x^s\sum_{i=1}^{\infty}2^{-si} = \frac{x^s}{2^s-1}.\]
   Hence, for $r=(2^s-1)^{-1} \in (0,1)$, we have $g(x) < rx^s < x^s$.
   Thus, by iterating, we have for all $k \in \mathbb{N}$
   \[g^k(x) < x^{s^k}.\]
   Suppose $x \in (0,\delta_1)$ and $k \in \mathbb{N}$ is the minimum value such that $g^{k-1}(x)\leq\delta_1 < g^{k}(x)$. 
   This gives us
   \[\log(\delta_1) \leq g^k(x) < s^{k}\log(x).\]
  Thus,
   \begin{equation}\label{B1}
      b(x) = |\log(x)|< |\log(\delta_1)|s^{-k}. 
   \end{equation}
   Let $E \coloneq \{j \in \{1,...,n-1\} : a_{0j}=1\}$, and $M$ from (\ref{M}).
   Then,\footnote{$M$ is the least integer such that $g^M(\delta_1) \notin I_0$ and $\ell = M-i \in \{M,M-1\}$ is the least integer such that $g^{\ell}(g^k(x)) \notin I_0$.} \[\pi^{-1}(x) \in\bigcup_{i=0,1} \bigcup_{j \in E }[0^{(k+M-i)}j].\]

 Since $b$ is bounded on $(I-p)\setminus(0,\delta_1)$, it is enough to show that $\int_{(0,\delta_1)} b(x)\d\mu(x) < \infty.$
 By Lemma \ref{Gibbs} and \eqref{B1}, and with $c_2 >0$ as in \eqref{GibbsBound}, we have
 \[\int_{(0,\delta_1)} b(x)\d\mu(x) \leq c_2|E||\log(\delta_1)|\sum_{i=0,1}\sum_{k=1}^{\infty} (e^{h_g})^{i-(k+M+1)}s^{-k} .\]
   Since $M$ is constant, this sum will converge if $s > e^{-h_g}$, which is guaranteed by our choice of $\epsilon$. 
   Therefore, the MME for $(I_K,g)$ is adapted, so by Lemma \ref{adapted}, the MME for $(I_J,f)$ is adapted.

 \subsubsection{Proof of statement (2) of Theorem \ref{thrm2}}
  Suppose $\underline\beta < 1$ and $h < -\frac{1}{n}\log(1-\underline{\beta})$. Then, $1-\underline\beta < e^{-h_g}$. Let $\epsilon>0$ be chosen such that $1-(\underline\beta-\epsilon) < e^{-h_g}$.
 By the Mean Value Theorem, for every $x \in (0, \delta_1)$, there exits a $c \in (0,x)$ such that $g'(c)x = g(x)$. 
   Hence, by \eqref{bound}, 
   \[x^{-(\underline\beta-\epsilon)} < c^{-(\underline\beta-\epsilon)} < g'(c) = \frac{g(x)}{x}.\]
So by setting $t=1-(\underline\beta - \epsilon)< e^{-h_g}$ and iterating, we have $x^{t^k}<g^k(x)$ as long as $g^k(x) \in (0,\delta_1)$. 
   Suppose $x \in (0,\delta_1)$ and $k \in \mathbb{N}$ is the minimum value such that $g^k(x) < \delta_1 < g^{k+1}(x) .$ 
   This gives us that $t^k\log(x)<\log(\delta_1) $ which in turn implies  
   \begin{equation}\label{B2}
       b(x) = |\log(x)|> t^{-k}|\log(\delta_1)|.
   \end{equation}
    We also have, for $E \coloneq \{j \in \{1,...,n-1\} : a_{0j}=1\}$ and $M$ from (\ref{M}),\footnote{The value $\ell = M+i \in \{M, M+1\}$ is the least value such that $g^{\ell}(g^k(x)) \notin I_0$.}
   \[\bigcup_{k \in \mathbb{N}}\bigcup_{i=0,1}\bigcup_{j \in E}[0^{k+M+i}j] \subset \pi^{-1}\big((0,\delta_1)\big).\]
   For any fixed $j \in E$, this will code the points in $(0,\delta_1)$ that stay in $(0,\delta_1)$ for at least $k$ iterates of $g$ and whose orbit will next intersect the interval coded by $j$.
   Thus, for a fixed $j \in E$ and $i=0$, by Lemma \ref{Gibbs} and \eqref{B1}, and with $c_1 >0$ as in \eqref{GibbsBound}, we have
   \[\int b(x) \d \mu(x) \geq  c_1\sum_k |\log(\delta_1)| (e^{h_g})^{-(k+M+1)}t^{-k}.\]
   Since $M$ is fixed, this sum will diverge if $t < e^{-h_g}$, which our choice of $\epsilon$ guarantees. 
   Therefore, the MME for $(I_K,g)$ is nonadapted so by Lemma \ref{adapted}, the MME for $(I_J,f)$ is nonadapted.

\subsubsection{Proof of statement (3) of Theorem \ref{thrm2}}
  Suppose $\overline\beta = -1$. 
Since $h < \infty$, we have $0 = \overline\beta +1 < e^{-h_g}$. 
Let $\epsilon >0$ be chosen such that $\epsilon < e^{-h_g}$. 
Letting $t = \epsilon$ allows the rest of the proof of statement (2) of Theorem \ref{thrm2} to apply here verbatim.

\subsection{Proof of Theorem \ref{thm:three}}\label{Pthrm3}
 Let $f \colon I \to I$ be a piecewise $C^1$ uniformly expanding Markov map that satisfies the three conditions of Theorem \ref{thm:three}. 
   Let $\mu$ be an $f$-invariant Borel probability measure on $I$. 
   Define the minimum subinterval length by 
   \begin{equation}
   \ell \coloneq \min_{1 < i \leq m}\{x_i-x_{i-1}\}.
   \end{equation}
   Since $f$ is Markov, $\{\pi_1(\tilde{f}^n(p^+))\}_{n \in \mathbb{N}} \subset B$.
  Let $\delta = \frac{\ell}{2}$.
    Since $f$ is H\"{o}lder continuous on $B_p \coloneq [p, p +\delta]$, there exists $C_H \geq 1,$ and $\alpha > 1$ such that for $\pi_1$ as in \eqref{proj},
    \[|\pi_1(\tilde{f}(x)) - \pi_1(\tilde{f}(p^+))| <C_H|x-p|^{1/\alpha}\]
    for $x \in [p, p+ \delta]$. 
    Away from the singularity, $|f'|$ is bounded so each $\tilde{f}_i$ is Lipschitz.
Hence, there is a $C_L >1$  such that for each $i \in \{0, ..., m-1\}$ and $y, z \in I_i \setminus B_p$ we have 
\[|\pi_1(\tilde{f}(z))-\pi_1(\tilde{f}(y))|\leq C_L|z-y|.\]

  The main strategy of this proof will be to show that since $p^+$ is not periodic, points in $B_p$ close to $p$ will have orbits that stay close to the orbit of $p^+$ and thus do not reenter $B_p$ for some controlled amount of iterates. 
  To explicitly control this amount, let us partition the subinterval 
   $[p, p+ (\delta C_H^{-1})^{\alpha}] \subset B_p$ into exponential subintervals, where for each $k \in \mathbb{N}$, 
   \begin{equation}\label{deekay}
   D_k \coloneq p + (\delta C_H^{-1})^{\alpha}[C_L^{-\alpha k}, C_L^{-\alpha(k-1)}].
    \end{equation}
    Thus, if $x \in D_k$, 
    \[|f(x)-\pi_1(\tilde{f}(p^+))| \leq C_H|x-p|^{1/ \alpha} \leq C_H \delta C_H^{-1}C_L^{-k} = \delta C_L^{-k}< \delta.\]
    Also, for $2 \leq i \leq k$, 
      \[|\pi_1(\tilde{f}^i(x))-\pi_1(\tilde{f}^i(p^+))| \leq \delta C_L^{i-1-k} \leq \delta.\]
This shows for $x \in D_k$, the $\tilde{f}$ orbit of $x$ and $p^+$ project by $\pi_1$ to the same subinterval for at least $k$ iterates.
Thus, by the choice of $\delta$, we have achieved our main goal\footnote{Note that it may be that $\pi_1(\tilde{f}^i(p^+)) = p$ as in the example in Figure \ref{fig:counter}. However, if this happens, the definition of $\tilde{f}$ \eqref{tildemap} implies that $\tilde{f}^i(p^+) = p^-$ is the right endpoint of the subinterval. Thus, the subinterval containing $\pi_1(\tilde{f}^i(x))$ is not the subinterval containing $B_p$.} of showing
\begin{equation}\label{orbitout}
 \text{for } x \in D_k \text{ and }  1 \leq i \leq k, \quad \pi_1(\tilde{f}^i(x)) \notin B_p.
\end{equation}
Let $b \colon B_p \to \mathbb{R}$ be defined by $b(x) = |\log(x-p)|$ and $b_k \coloneq b|_{D_k}$.
Then, 
\[b_k(x) \leq \alpha |\log(\delta C_H^{-1} C_L^{-k})| = \alpha [ \log(C_H \delta^{-1}) + k\log(C_L)] < \infty,\]
so $b_k \in L^1(I,\mu)$.
    We have
    \begin{equation}\label{bkint}
     \int b(x) \d \mu(x) = \sum_{k=1}^{\infty} \int b_k(x) \d \mu(x).
    \end{equation}
   By \eqref{orbitout}, if $0 \leq i \leq n$, then $b_k \circ f^i(x) = 0$ for all but at most $\lfloor \frac{n}{k+1} \rfloor + 1$ values of $i$. 
    Hence, the Birkhoff averages for $b_k$ are bounded in the following way
    \[ \frac{1}{n}S_n(b_k)(x)\coloneq \frac{1}{n} \sum_{i=0}^{n-1} b_k \circ f^i(x) \leq \frac{\alpha}{n}\bigg[ \log(C_H \delta^{-1}) + k\log(C_L)\bigg]\cdot\left( \frac{n}{k+1} +1\right) \coloneq M_{n,k}.\]
    Hence,
    \begin{align}\label{Mbound}
    \begin{split}
    \lim_{n \to \infty} \frac{1}{n}S_n(b_k)(x) \leq \lim_{n \to \infty}M_{n,k} &= \frac{\alpha}{k+1}\big[ \log(C_H \delta^{-1}) + k\log(C_L)\big] \\ &\leq \alpha\log(C_LC_H\delta^{-1}).
    \end{split}
    \end{align}
    Thus, by the Birkhoff Ergodic Theorem \cite[Theorem 4.5.5]{BS}, 
    \[ \int_I b_k(x) \d \mu(x) = \int_{D_k} \lim_{n \to \infty} \frac{1}{n}S_n(b_k)(x) \d \mu(x) \leq \alpha\log(C_LC_H\delta^{-1}) \mu(D_k). \]
    Therefore, by \eqref{bkint} and \eqref{Mbound}, we have
    \[\int b(x) \d \mu(x) = \sum_{k=1}^{\infty} \int b_k(x) \d \mu(x) \leq  \alpha\log(C_LC_H\delta^{-1})\sum_{k=1}^{\infty}  \mu(D_k) < \infty,\]
    so $\mu$ is adapted.

\section{Examples and Applications}\label{examples}

\subsection{Examples}
If the topological entropy falls in the range in \eqref{indeterm}, then Theorem \ref{thrm2} does not determine whether the MME for $(I_J,f)$ is adapted or not.
Recall from Remark \ref{threshold} that if $f(x) = x^{1/\alpha}$ near a fixed singularity at $0$, $\underline\beta = \overline\beta = \frac{1}{\alpha} -1$. 
In this case, the value for the entropy, $h = \frac{1}{n}\log(\alpha)$, is in the indeterminate interval, and the MME could be either adapted or nonadapted.
To demonstrate this, we will show two examples.

\begin{example}\label{eqnonadapt}
There exists an interval map satisfying the conditions of Theorem \ref{thrm2} such that $h = \log(\alpha)$ and the MME is nonadapted.
\end{example} 
\begin{proof}
Let $f\colon I \to I$ be a uniformly expanding Markov map conjugate to the doubling map, $T(x) =2x \mod(1)$, such that 
\[f|_{\big[0,\frac{1}{16}\big]}(x) = \sqrt{x}.\] 
See (A) in Figure \ref{fig:examples} as an example.
Then, $f$ satisfies the conditions of Theorem \ref{thrm2}, has a fixed singularity at $0$, is semiconjugate to the one-sided shift on two elements, and the MME of $(I,f)$ is a Bernoulli measure with entropy $\log(2).$
We calculate, from the definition \eqref{loglim}, $\underline\beta = \overline\beta = \beta$ by the following
\[\lim_{x \to 0^+}\frac{\log((f(x))')}{-\log(x)} = \lim_{x \to 0^+}\frac{-\log(2)-\frac{1}{2}\log(x)}{-\log(x)} =\lim_{x \to 0^+}\frac{\log(2)}{\log(x)} +\frac{1}{2} = \frac{1}{2}.   \]
   Thus, $\alpha = 2 = e^{h}$.
That the MME is nonadapted follows from Proposition \ref{propeq}.
\end{proof}

\begin{example}\label{eqadapt}
There exists a map satisfying the conditions of Theorem \ref{thrm2} such that $\log(\alpha) = h$ and the MME is adapted.
\end{example}
\begin{proof}
Let \[ g(x) =\frac{1}{ \log(\log(|\log(x)|))}\]
and $0<\rho < e^{-e^e},$ so $g(\rho)<1$.
Suppose $f\colon I \to I$ is a uniformly expanding Markov map conjugate to the doubling map, $T(x) =2x \mod(1)$, such that $f(0)=0$ and $f|_{(0,\rho)}(x) = x^{\frac{1}{2-g(x)}}$. 
Then $f$ has a fixed singularity at $0$, is semiconjugate to the one-sided shift on two symbols, and the MME of $f$ is a Bernoulli measure with an entropy of  $\log(2).$

First, let us check that $f$ satisfies the limit condition
\[\lim_{x \to 0^+}\frac{\log(f'(x))}{-\log(x)} = \frac{1}{2}.\]
Note $\lim_{x \to 0^+}g(x) = 0$ and $g(x)>0$ on $(0,\rho).$
To find $f'(x)$ on $(0,\rho)$ we take 
\[\log(f(x)) = \frac{\log(x)}{2-g(x)}.\]
Taking a derivative we have
   \[\frac{f'(x)}{f(x)} = \frac{(2-g(x))x^{-1} + g'(x)\log(x)}{(2-g(x))^2}.\] 
Let $z(x) \coloneq \left(2-g(x) + xg'(x)\log(x)\right) x^{-1},$ so
 \[
     \log(f'(x)) = \log(z(x)) - 2\log(2-g(x)) + \log(f(x)) .\]
 Therefore,
     \[
     \frac{\log(f'(x))}{-\log(x)} = \frac{\log(z(x))}{-\log(x)}+2\frac{\log(2-g(x))}{\log(x)} - \frac{1}{(2-g(x))}.\]
Note that \[g'(x) = -(g(x))^2\frac{1}{\log(|\log(x)|)}\frac{1}{\log(x)}\frac{1}{x}>0 \quad \text{on} \quad (0,\rho),\] and $\lim_{x \to 0^+} x\log(x)g'(x) = 0$.
Computing the limits separately, we first have 
\[\lim_{x \to 0^+}\frac{\log(z(x))}{-\log(x)} =\lim_{x \to 0^+}\frac{\log\left(2-g(x)+xg'(x)\log(x)\right) - \log(x)}{-\log(x)} =1. \]
The other two are clear by inspection, so we have
\[\lim_{x \to 0^+}\frac{\log((f(x))')}{-\log(x)} = 1 -0-\frac{1}{2} = \frac{1}{2}.\]
This shows $\alpha = 2 = e^h.$

We now show the MME for $(I,f)$ is adapted.
Suppose $x \in (0, \rho)$ and $m \in \mathbb{N}$ is the minimum value such that $f^m(x) \leq \rho< f^{m+1}(x)$.
Then, $ f(x) =x^{1/(2-g(x))}$ and 
\[f^2(x) =\left(x^{\frac{1}{(2-g(x))}}\right)^{\frac{1}{(2-g(f(x)))}}.\]
Since $f(x)>x$ and $g'(x)>0$ we have that $2-g(f(x)) < 2-g(x)$ so
\[f^2(x) < x^{(2-g(x))^{-2}}.\]
Repeating this argument we have that
\[\rho < f^{m+1}(x) < x^{(2-g(x))^{-m-1}}.\]
Thus,
\[\log(\rho) < (2-g(x))^{-m-1}\log(x),\]
so
\begin{equation}\label{bfunk}
b(x) = |\log(x)| < |\log(\rho)|(2-g(x))^{m+1}.
\end{equation}
We also need a bound for $(2-g(x))$ so by noting that $f(x)< \sqrt x$ we have that \[\rho < f^{m+1}(x) < x^{2^{-m-1}}\]
which implies 
\begin{equation}
  (2-g(x))< 2-g(\rho^{2^{m+1}}) = 2-\frac{1}{\log\big[(m+1)\log(2)+\log(|\log(\rho)|)\big]}\coloneq \eta_m.  
\end{equation} 
Thus, by \eqref{bfunk} and with $c_2>0$ as in \eqref{GibbsBound}, we have
\begin{gather}\label{series}
\int b(x) \d \mu(x) < c_2|\log(\rho)|\sum_m 2^{-m}\eta^{m+1}_m 
\end{gather}
which is a convergent series by the following argument.

Consider the series $\sum_m (1-\frac{r}{\log(m)})^m$ for any $r>0$.
By the Cauchy Condensation Test, this series will converge if 
\begin{equation}\label{CCT}
    \sum_m a_m = \sum_m 2^m\left(1-\frac{r}{m\log(2)}\right)^{2^m}
\end{equation}
converges.
Note that 
\[\lim_{m \to \infty} \frac{1}{m}\log\left(a_m\right) =\log(2) + \lim_{m \to \infty}\frac{2^m}{m} \log\left(1-\frac{r}{m\log(2)}\right)= -\infty.\]
Thus, $a_m$ has superexponential decay.
Hence, the series (\ref{CCT}) converges.

To return to the series in $(\ref{series})$, let $r = \frac{1}{4}$ and note there exists an $M \in \mathbb{N}$ such that for $m > M$ 
\[2 \log\big[(m+1)\log(2)+\log(|\log(\rho)|)\big] < 4\log(m+1).\]
Thus, for $m>M$,
\begin{align*}
    2^{-m}\eta^{m+1}_m &= 2\left(1- \frac{1}{2\log\big[(m+1)\log(2)+\log(|\log(\rho)|)\big]}\right)^{m+1} \\
    &< 2\left(1-\frac{r}{\log(m+1)}\right)^{m+1}.
\end{align*}
Therefore, the series bounding $\int b(x) \d \mu(x)$ in (\ref{series}) converges.
Thus, the MME for $(I,f)$ is adapted.
\end{proof}
Finally, we show that if we do not require the map in Theorem \ref{thm:three} to be H\"{o}lder continuous near the singularity then it may be that not all invariant measures are adapted.

\begin{example}\label{nonpolynonadapt}
  There exists a uniformly expanding Markov interval map $f \colon I \to I$ satisfying the following:
  
  \begin{itemize}
      \item $f$ has exactly one singularity, $\frac{1}{2}^+$,
      \item $\frac{1}{2}^+$ is not periodic,
      \item $f$ is not H\"{o}lder continuous on $[\frac{1}{2},\delta]$ for any $\delta >\frac{1}{2}$,
      \item the MME of $(I,f)$ is nonadapted.
  \end{itemize}
\end{example}
Let $I = [0, 1]$.
Consider the Markov function $f \colon I \to I$ defined on two subintervals by
\[ \begin{cases} 
      f_0(x) = 2x & x \in I_0 = [0,\frac{1}{2}] \\
      f_1(x) = -\frac{\log(2)}{\log(x-\frac{1}{2})} & x \in I_1 = (\frac{1}{2}, 1] \\
      f_1(\frac{1}{2})=0.
   \end{cases}
\]
Thus, $f$ is uniformly expanding and $f_1$ has an inverse $g(x) = 2^{-1/x}+\frac{1}{2}$ for $x \in (0,1]$ and $g(0) = \frac{1}{2}$.
Let $q = \frac{1}{2}$.
Then, $f$ is not H\"{o}lder continuous on $[q,1]$ and $f(q) = 0 = f^2(q)$, so $q $ is not periodic.
The map, $f$, is semiconjugate to the full one-sided shift on two symbols and the the MME of $(I,f)$ is a Bernoulli measure which has an entropy of $\log(2)$. 
Suppose $x>q$ is coded by a sequence in $[10^n1]$ so $f(x)< q$ and $n\in \mathbb{N}$ is the minimum value such that $f^{n}(x)<q<f^{n+1}(x)$.
Then, \[2^{n-1}f(x) < q < 2^nf(x), \text{ so } f(x) < q^n.\] 
Thus, $x < g(q^n)+q = 2^{-1/q^n}+q$, so $x-q < 2^{-1/q^n}$.
Hence, 
\[b(x) = |\log(x-q)| > 2^{n}\log(2) .\]
Thus, if $\mu$ is the MME for $(I,f)$, 
\[\int b(x) \d \mu(x) > \log(2) \sum_{n=1}^{\infty}\mu([10^n1])2^n =  \log(2)\sum_{n=1}^{\infty} 2^{-n-2}\cdot2^n = \infty. \]
Therefore, the MME is nonadapted.

\subsection{Dimension of Ergodic Measures}
Consider the following definition of dimension from \cite{Led}, which is related to the upper box dimension.
Let $N(\epsilon,\delta,\mu)$ be the minimal number of balls of radius $\epsilon$ needed to cover a region of the interval with measure greater than $1-\delta$.
Then, define 
\[\dim(\mu) \coloneq \lim_{\delta \to 0} \limsup_{\epsilon \to 0} \frac{N(\epsilon, \delta, \mu)}{-\log(\epsilon)}.\]
Let $f$ be a transitive interval map satisfying the conditions of Theorem \ref{thrm2} on $I$ such that $f'$ is monotonic on each subinterval, and, for simplicity, assume $f$ has a fixed singularity at $0$.
Let $\mu$ be the MME of $(I,f)$.
Finally, suppose $\int_I \log(|f'|)\d\mu >0$. 
Then, \cite[Proposition 4]{Led} states
\begin{equation}\label{dimension}
\dim(\mu) = \frac{h(\mu)}{\int_I \log(f'(x))\d\mu(x)}.
\end{equation}
Recalling the values $\underline \beta, \overline\beta$ from \eqref{loglim} we have the following which we prove below.
\begin{itemize}
    \item[(A)] If $\underline{\beta} > 0$ and $h(\mu) >0$, then $\mu$ is nonadapted if and only if $\dim(\mu)= 0$.
    \item[(B)] If $h(\mu) >0$ and $\mu$ is adapted, then $\dim(\mu) >0$.
    \item[(C)] If $\underline{\beta} = 0$, then it is possible for a nonadapted measure to have $\dim(\mu) >0$.
\end{itemize}
If $\underline\beta > 0$ and $h(\mu)>0$, then there exist $\epsilon, \delta>0$ such that for $x \in (0, \delta),$
    \[0< \underline\beta - \epsilon < \frac{\log(f'(x))}{-\log(x)} < \overline\beta + \epsilon .\]
This implies 
\begin{equation}\label{dimmu}
    -\log(x)(\underline\beta - \epsilon) < \log(f'(x)) < -(\overline\beta +\epsilon)\log(x).
\end{equation}
Statement (A) shows that nonadapted measures are highly concentrated near the singularity.
It also implies that the interval maps with nonadapted MMEs we have been considering give examples where the Lyapunov exponent is infinite.
Neil Dobbs, in \cite[Section 11]{ND}, also examined interval maps with a parameter describing a singularity and gave examples where the MME has an infinite Lyapunov exponent.

To show (A), note if $\mu$ is nonadapted, then the integral of the left most term in \eqref{dimmu} is infinite. Hence, $\int_I \log(f'(x))\d\mu(x) = \infty$, so $\dim(\mu) = 0$ by \eqref{dimension}.
If $\dim(\mu) = 0$, $\int \log(f'(x))\d\mu(x) = \infty$ so by the right inequality in  \eqref{dimmu}, $\mu$ is nonadapted.

For (B), assume for a contradiction that $\dim(\mu) = 0.$
Since $h(\mu) \leq h_\mathrm{top} <\infty,$ we must have $\int\log(f')\d\mu = \infty$.
For any $c>1$, if there exists a $\delta>0$ such that for all $x \in (0,\delta)$, $f'(x) < x^{-c}$, then \[\int\log(f'(x))\d\mu(x) < \int c |\log(x)|\d\mu(x) = \int |\log(x^c)|\d \mu(x)<\infty,\]
where the last inequality uses adaptedness.
Thus, there exists a decreasing sequence in $(0,1)$, $\{x_i\}_{i \in \mathbb{N}}$, such that $f'(x_i) \geq x_i^{-i}.$
Furthermore, by monotonicity of $f'$, $f'(x) \geq x_i^{-i}$ for $x \in (x_{i+1},x_i]$.
Since $f(x_1) \in [0,1]$, we must have 
\begin{equation}\label{sumbound}
    1 \geq f(x_1) = \int_{(0,x_1)}f'(x)\d x \geq \sum_{i=1}^{\infty}(x_i-x_{i+1})x_i^{-i}.
\end{equation}
We must also have that $\sum_{i=1}^{\infty}(x_i-x_{i+1}) = x_1$.
We will show this is impossible.
Suppose by way of contradiction that the above holds for some sequence $\{x_i\}_{i \in \mathbb{N}}$.
Then, there exist natural numbers $N, K >2$ such that $x_{K} <  2^{-N} < x_2$.
Also, since the sum in \eqref{sumbound} is less than $1$, for all $i > K$ we have $(x_{i}-x_{i+1})x_{i}^{i} < 1$ and $x_i^i < x_{K}2^{-N(i-1)}$.
Hence, 
\begin{equation}
    \sum_{i = K+1}^{\infty}(x_i-x_{i+1}) < \sum_{i = K+1}^{\infty}x_{K}2^{-N(i-1)} = x_{K}\frac{2^{-NK}}{1-2^{-N} }< x_{K}.
\end{equation}
This contradicts the assumption that $x_i \to 0$.

(C) However, if $\underline\beta=0$, it is possible for a nonadapted measure to have positive dimension.
This does not even depend on the strength of the singularity. 
Consider the doubling map $f(x) = 2x \mod 1$. 
Here, $0$ is not a singularity according to Definition \ref{singularity}, but if we take it to be the singular point in Definition \ref{adaptedmeasure}, it is still possible to construct a $0$-nonadapted measure with positive entropy.
Since the Lyapunov exponent here is $\log(2)$, this will mean $\dim(\mu)>0$.
We will construct such an ergodic measure using a return map and a full shift on a countable alphabet.

\begin{example}\label{nonsing}
   Let $f\colon [0,1] \to [0,1]$ be defined by $f(x) = 2x\mod 1$. 
   Then, there exists an ergodic $f$-invariant $\nu$ such that $h(\nu)>0$ and $\nu$ is nonadapted with respect to $0$.
\end{example}
    
Let $X = \{0,1\}^{\mathbb{N}}$ and $Y =[1] \subset X$.
Let $\tau \colon Y \to \mathbb{N}$ be defined by $\tau(\omega) = n$ for all $\omega \in [10^{n-1}1]$, where $n \in  \mathbb{N}$.
Let $T \colon Y \to Y$ be defined by $x \mapsto \sigma^{\tau(x)}(x)$.
Let $Z =\mathbb{N}^\mathbb{N}$ and let $\overline \sigma$ denote the left shift on $Z$.
Define $\eta \colon Z \to Y$ by 
\[\eta(n_1n_2n_3...) = (10^{n_1-1}10^{n_2-1}10^{n_3-1}1...).\]
Thus, $\eta \circ \overline\sigma = T \circ\eta$.
Let $p \in [0,1]^{\mathbb{N}}$ be such that $\sum_{i \in \mathbb{N}} p_i = 1$ and $p_i \neq 0$ for all $i \in \mathbb{N}$.
Let $m$ be the Bernoulli measure on $Z$ defined by $p$.
Note that 
\begin{equation}
h_{\overline\sigma}(m) = \sum_{i \in \mathbb{N}}-p_i\log(p_i) >0.
\end{equation}
On $Y$ we have the pushforward $m_0 \coloneq \eta_*m$.
Let us define 
\[\tilde{\mu} \coloneq \sum_{n=0}^{\infty}\sigma_*^nm_0|_{\tau > n}, \]
that is, for  $\tilde\mu$ measurable $A$,
\[\tilde\mu (A) = \sum_{n=0}^{\infty}\sum_{k=n+1}^{\infty}m_0(\sigma^{-n}(A)\cap[10^{k-1}1]),\]
and $\mu = \tilde{\mu}/\tilde{\mu}(X)$.
Hence, $\mu$ will be a $\sigma$-invariant ergodic measure on $X$. 

Then, since $h_{T}(m_0) = h_{\overline\sigma}(m)$, by Abramov's formula \cite[Section 6.1.C]{Pet} we have
\[ h_{\sigma}(\mu) =h_T(m_0)\mu(Y) = h_{\overline\sigma}(m)\frac{\tilde\mu(Y)}{\tilde\mu(X)}. \]
Since \[\tilde\mu(Y) = \sum_{n=0}^{\infty}\sum_{k=n+1}^{\infty}m_0(\sigma^{-n}(Y)\cap[10^{k-1}1]) = \sum_{n=0}^{\infty}m_0([10^{n}1]) = \sum_{n=0}^{\infty}p_{n+1}=1,\]
and \[\tilde\mu(X) = \sum_{n=0}^{\infty}\sum_{k=n+1}^{\infty}m_0(\sigma^{-n}(X)\cap[10^{k-1}1])=\sum_{n=0}^{\infty}\sum_{k=n+1}^{\infty}m_0([10^{k-1}1]) = \int_Y\tau(\omega) \d m_0(\omega) ,\]
we can rewrite Abramov's formula as
\[
   h_{\overline\sigma}(m)\frac{\tilde\mu(Y)}{\tilde\mu(X)} = \frac{h_{\overline\sigma}(m)}{\int_Y\tau(\omega) \d m_0(\omega)}.
\]
Thus, $h_{\sigma}(\mu) > 0 $ if $\int_Y\tau(\omega) \d m_0(\omega)<\infty$.

To achieve this, while also ensuring $\mu$ will be nonadapted, let us define $p_n = \frac{c}{n^3}$ where $c =(\sum_{n=1}^{\infty}\frac{1}{n^3})^{-1} $ is the normalizing constant.
Thus, $\sum_{n=1}^{\infty} p_n = 1$ and 
\[ \int_Y\tau(\omega)\d m_0(\omega) = \sum_{n=1}^{\infty}np_n < \infty.\]
We now show that $\nu \coloneq \pi_*\mu$, where $\pi$ is the projection onto $I$ defined in Section \ref{coding}, is nonadapted.
Note that since $f(x) =2x \mod 1$, if $x \in (2^{-n-1},2^{-n})$, then $b(x) > n\log2$ and $\pi^{-1}(x) \subset [0^n1]$.
Thus, for $\nu$ to be nonadapted, we must show $\sum_{n=1}^{\infty}n\mu([0^n1]) = \infty.$
By invariance, $\mu([01]) = \mu([101])+\mu([0^21]) = p_2+ \mu([10^21]) + \mu([0^31])$, and so on.
Thus, 
\begin{equation}
\mu([0^n1]) = \sum_{i=n+1}^{\infty}p_i=\sum_{i=n+1}^{\infty}\frac{c}{i^3} \geq \frac{c}{2(n+1)^2}.
\end{equation}
Hence, $\sum_{n=1}^{\infty}n\mu([0^n1]) \geq \sum_{n=1}^{\infty}\frac{cn}{2(n+1)^2} = \infty.$
Therefore, $\nu$ is nonadapted.

\subsection{Interval Maps from the Geometric Lorenz Models}\label{Lorenz}

To conclude, we apply our results to interval maps induced by geometric Lorenz models.
These models were introduced in the 70s independently by V. S. Afraimovich, V. V. Bykov and L. P. Shil’nikov \cite{ABS} and by  Guckenheimer and Williams \cite{GW}.
There were motivated by the Lorenz flow.
For details, see \cite{GP}. They give a construction 
starting with the flow $(\dot x, \dot y, \dot z) = (\lambda_1x, \lambda_2y, \lambda_3z)$ on $[-1,1]^3$ such that 
\begin{equation}
    0 < \frac{\lambda_1}{2} \leq -\lambda_3 < \lambda_1 < -\lambda_2,
\end{equation}
and the Poincar\'{e} first return map to $[-\frac{1}{2}, \frac{1}{2}]^2 \times\{1\}$ induces a skew product $F \colon [\frac{1}{2}, \frac{1}{2}]^2 \to [\frac{1}{2}, \frac{1}{2}]^2$ of the form $F(x,y) = (f(x), g(x,y))$. 
Here, the Lorenz map, $f$, is odd, piecewise expanding, and for all $x$, the map $y \mapsto g(x,y)$ is contracting. 
Thus, the geometric Lorenz model is a suspension flow over the natural extension of $f$.

Since the Lorenz map is odd, we will consider a related function scaled to $[0,1]$, $f_a$, defined by $f_a(x) = 2|f(\frac{x}{2})|$ for $x > 0$.
Even though $f$ was not defined at $0$, $\lim_{x \to 0}|f(x)| = 1$ so we may define $f_a(0)=1$.
We record some facts about $f_a$ from \cite{GP}:
\begin{enumerate}
    \item $f_a$ is $C^1$ on $(0,1)$ except at the point $b \in (0,1)$ satisfying $f(b)=0$,
    \item $|f_a'(x)| = Cx^{B - 1}>1$ where $B = -\frac{\lambda_3}{\lambda_1} \in (0,1)$ and $C>0$,\footnote{See (14) on page 1710 in \cite{GP}.}
 \item $f_a(1)< 1.$
\end{enumerate}
By (2), $f_a$ is uniformly expanding and the limits in Theorem \ref{thrm2} coincide.
That is, $\underline\beta = \overline\beta =\beta$ \eqref{loglim} can be calculated to be $\beta = 1-B  =1- \frac{1}{\alpha}$ which also gives us $\alpha = - \frac{\lambda_1}{\lambda_3}>1$.

For any geometric Lorenz model as constructed by Galatolo and Pacifico, $f_a$ would have these properties.
However, in order to use Theorem \ref{thrm2}, we need $f_a$ to be a Markov map. 

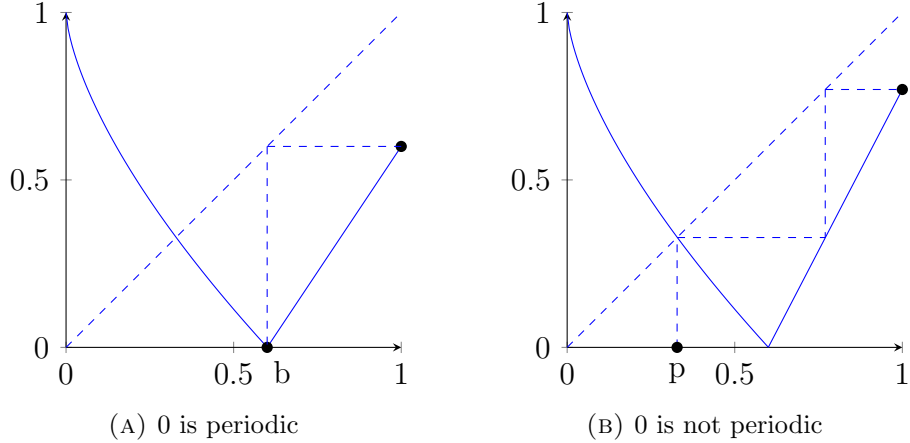
\begin{figure}
  \centering
  \begin{subfigure}{.35\linewidth}
    \centering
    \begin{tikzpicture}
\begin{axis}[axis lines = left, axis equal image, width = 7cm, xtick={0,.5,1},
    ytick={0,.5,1}, clip = false,]

    \addplot [domain=0:.6, samples=100, color=blue,]{-1.4*x^.66 +1};
    \addplot[color=blue, mark=none,dashed,] coordinates{(0,0) (1,1)};
    \addplot[color=blue, mark=none,smooth,] coordinates{(1,.6) (.6,0)};
    \node[label={[xshift=.2cm, yshift=-0.7cm]b}, circle,fill,inner sep=1.5pt] at (axis cs:.6,0) {};
    \node[circle,fill,inner sep=1.5pt] at (axis cs:1,.6,) {};
    \addplot[color=blue, mark=none,dashed,] coordinates{(.6,.6) (.6,0)};
    \addplot[color=blue, mark=none,dashed,] coordinates{(.6,.6) (1,.6)};
\end{axis}
\end{tikzpicture}
    \caption{0 is periodic}
  \end{subfigure}%
  \hspace{3em}% 
  \begin{subfigure}{.35\linewidth}
    \centering
    \begin{tikzpicture}
\begin{axis}[axis lines = left, axis equal image, width = 7cm, xtick={0,.5,1},
    ytick={0,.5,1}, clip = false,]

    \addplot [domain=0:.6, samples=100, color=blue,]{-1.4*x^.66 +1};
    \addplot[color=blue, mark=none,dashed,] coordinates{(0,0) (1,1)};
    \addplot[color=blue, mark=none,smooth,] coordinates{(1,.77) (.6,0)};
    \node[circle,fill,inner sep=1.5pt] at (axis cs:1,.77,) {};
    \addplot[color=blue, mark=none,dashed,] coordinates{(.77,.77) (1,.77)};
    \addplot[color=blue, mark=none,dashed,] coordinates{(.328,0) (.328,.328)};
    \addplot[color=blue, mark=none,dashed,] coordinates{ (.328,.328) (.77,.328)};
    \addplot[color=blue, mark=none,dashed,] coordinates{ (.77,.77) (.77,.328)};
    \node[label={[yshift=-0.7cm]p}, circle,fill,inner sep=1.5pt] at (axis cs:.328,0) {};
\end{axis}
\end{tikzpicture}
    \caption{0 is not periodic}
  \end{subfigure}%
  \caption{ Examples of $f_a$ Interval maps induced by Lorenz maps.}%
   \label{fig:Lorenz Map}%
\end{figure}

The only way that this could happen to fit the definition of a Markov map given above is if $b$ is eventually periodic.
Note the orbit of $b$ is $\{b, 0, 1, ...\}$.
If $F$ was constructed such that $b$ is eventually periodic, the orbit of $b$ would determine a partition $\{0=x_0 < x_1 <  ... < x_m=1\}$.
That is, $x_i= f_a^k(b)$ for some $k \geq 0$.
Then, the transitive component, $I_J$, of $f_a$ containing $[0,x_1]$ could be coded by a SFT as described in Section \ref{coding}. 
Denote the entropy of $(I_J,f_a)$ by $h$.
There are two ways that $b$ could be eventually periodic. The first is that $f_a^n(b) =b$ for some minimum $n \in \mathbb{N}$. 
An example is shown in Figure \ref{fig:Lorenz Map} graph (A).
This corresponds to the situation where the $\{x=1/2\}$ section of the geometric Lorenz flow is in the stable manifold of $0$.
Thus, by Theorem \ref{thrm2}, if $\alpha = - \frac{\lambda_1}{\lambda_3} < e^{nh}$, the MME for $f_a$ is adapted and if $\alpha = - \frac{\lambda_1}{\lambda_3} > e^{nh}$, the MME for $f_a$ is nonadapted.
 
Otherwise, the orbit of $b$ is eventually periodic to some other point in $[0,1] \setminus \{0,b,1\}$.
In this case, the singularity would not be periodic. 
Then, since (2) implies $f$ is H\"{o}lder continuous near 0, Theorem \ref{thm:three} implies every invariant Borel probability measure for $f_a$ is adapted.
One example is that the orbit of $b$ hits the fixed point of $f_a$. 
That is, since the graph of $f$ intersects the line $y=-x$ at some $p>0$, $f_a$ will have a single fixed point. 
An example is shown in Figure \ref{fig:Lorenz Map} graph (B).
This would indicate that $\{x=1/2\}$ section of the geometric Lorenz flow enters a periodic flow forming a figure eight like shape.

\bibliographystyle{alpha}
\bibliography{Intervalmaps}

\begin{thebibliography}{CDLZ24}

\bibitem[ABS77]{ABS}
V.~S. Afraimovich, V.~V. Bykov, and L.~P. Sil'nikov.
\newblock The origin and structure of the {L}orenz attractor.
\newblock {\em Dokl. Akad. Nauk SSSR}, 234(2):336--339, 1977.

\bibitem[BD20]{BD}
Viviane Baladi and Mark~F. Demers.
\newblock On the measure of maximal entropy for finite horizon {S}inai billiard maps.
\newblock {\em J. Amer. Math. Soc.}, 33(2):381--449, 2020.

\bibitem[BS15]{BS}
Michael Brin and Garrett Stuck.
\newblock {\em Introduction to dynamical systems}.
\newblock Cambridge University Press, Cambridge, back edition, 2015.

\bibitem[CDLZ24]{CDLZ}
Vaughn Climenhaga, Mark~F. Demers, Yuri Lima, and Hongkun Zhang.
\newblock Lyapunov exponents and nonadapted measures for dispersing billiards.
\newblock {\em Comm. Math. Phys.}, 405(2):Paper No. 24, 13, 2024.

\bibitem[Dob14]{ND}
Neil Dobbs.
\newblock On cusps and flat tops.
\newblock {\em Ann. Inst. Fourier (Grenoble)}, 64(2):571--605, 2014.

\bibitem[GP10]{GP}
S.~Galatolo and Maria~Jos\'e Pacifico.
\newblock Lorenz-like flows: exponential decay of correlations for the {P}oincar\'e{} map, logarithm law, quantitative recurrence.
\newblock {\em Ergodic Theory Dynam. Systems}, 30(6):1703--1737, 2010.

\bibitem[GW79]{GW}
John Guckenheimer and R.~F. Williams.
\newblock Structural stability of {L}orenz attractors.
\newblock {\em Inst. Hautes \'Etudes Sci. Publ. Math.}, (50):59--72, 1979.

\bibitem[KH95]{KH}
Anatole Katok and Boris Hasselblatt.
\newblock {\em Introduction to the modern theory of dynamical systems}, volume~54 of {\em Encyclopedia of Mathematics and its Applications}.
\newblock Cambridge University Press, Cambridge, 1995.
\newblock With a supplementary chapter by Katok and Leonardo Mendoza.

\bibitem[KSLP86]{KS}
Anatole Katok, Jean-Marie Strelcyn, F.~Ledrappier, and F.~Przytycki.
\newblock {\em Invariant manifolds, entropy and billiards; smooth maps with singularities}, volume 1222 of {\em Lecture Notes in Mathematics}.
\newblock Springer-Verlag, Berlin, 1986.

\bibitem[Led81]{Led}
F.~Ledrappier.
\newblock Some relations between dimension and {L}yapounov exponents.
\newblock {\em Comm. Math. Phys.}, 81(2):229--238, 1981.

\bibitem[LM95]{LinMar}
Douglas Lind and Brian Marcus.
\newblock {\em An introduction to symbolic dynamics and coding}.
\newblock Cambridge University Press, Cambridge, 1995.

\bibitem[LM18]{LM}
Yuri Lima and Carlos Matheus.
\newblock Symbolic dynamics for non-uniformly hyperbolic surface maps with discontinuities.
\newblock {\em Ann. Sci. \'Ec. Norm. Sup\'er. (4)}, 51(1):1--38, 2018.

\bibitem[LS19]{LS}
Yuri Lima and Omri~M. Sarig.
\newblock Symbolic dynamics for three-dimensional flows with positive topological entropy.
\newblock {\em J. Eur. Math. Soc. (JEMS)}, 21(1):199--256, 2019.

\bibitem[Pet83]{Pet}
Karl Petersen.
\newblock {\em Ergodic theory}, volume~2 of {\em Cambridge Studies in Advanced Mathematics}.
\newblock Cambridge University Press, Cambridge, 1983.

\bibitem[PY98]{PY}
Mark Pollicott and Michiko Yuri.
\newblock {\em Dynamical systems and ergodic theory}, volume~40 of {\em London Mathematical Society Student Texts}.
\newblock Cambridge University Press, Cambridge, 1998.

\end{thebibliography}

\end{document}